\documentclass[12pt]{article}
\usepackage{a4wide}
\usepackage{amsmath}
\usepackage{amssymb}
\usepackage{booktabs}
\usepackage{comment}
\usepackage{latexsym}
\usepackage{mathpazo}
\usepackage{mathrsfs}
\usepackage{tikz}
\usepackage[textwidth=30mm]{todonotes}
\usepackage[normalem]{ulem}

\newtheorem{thm}{Theorem}

\newtheorem{defn}[thm]{Definition}
\newtheorem{exl}[thm]{Example}
\newtheorem{rem}[thm]{Remark}
\newtheorem{lemma}[thm]{Lemma}

\newcommand{\Eqref}[1]{Eq.\,(\ref{#1})}
\newcommand{\ep}{\epsilon}

\newcommand{\N}{\mathbb{N}}
\newcommand{\R}{\mathbb{R}}

\newcommand{\E}{\mathbb{E}}
\newcommand{\G}{\mathscr{G}}

\newenvironment{proof}{\vspace{-2pt}\noindent\textbf{Proof:}}{$\Box$\\\medskip}

\begin{document}
\title{Approximations of expectations under infinite product measures\thanks{Solan acknowledges the support of the Israel Science Foundation, grant \#211/22. We are grateful to John (Yehuda) Levy for a stimulating discussion.}}

\author{Galit Ashkenazi-Golan\footnote{London School of Economics and Political Science, Houghton Street, London WC2A 2AE, UK, E-mail: galit.ashkenazi@gmail.com.}\and 
J\'{a}nos Flesch\footnote{Department of Quantitative Economics, 
Maastricht University, P.O.Box 616, 6200 MD, The Netherlands. E-mail: j.flesch@maastrichtuniversity.nl.}
\and Arkadi Predtetchinski\footnote{Department of Economics, Maastricht University, P.O.Box 616, 6200 MD, 
The Netherlands. E-mail: a.predtetchinski@maastrichtuniversity.nl.}\and 
Eilon Solan\footnote{School of Mathematical Sciences, Tel-Aviv University, Tel-Aviv, Israel, 6997800, E-mail: eilons@tauex.tau.ac.il.}}
\maketitle

\begin{abstract}
\noindent We are given a bounded Borel-measurable real-valued function on a product of countably many Polish spaces, and a product probability measure. We are interested in points in the product space that can be used to approximate the expected value of this function. We define two notions. A point is called a \emph{weak $\ep$-approximation}, where $\ep\geq 0$, if the Dirac measure on this point, except in finitely many coordinates where another measure can be taken, gives an expected value that is $\ep$-close to the original expected value. A point is called a \emph{strong $\ep$-approximation} if the same holds under the restriction that in those finitely many coordinates the measure is equal to the original one. We prove that both the set of weak 0-approximation points and the set of strong $\ep$-approximation points, for any $\ep>0$, have measure 1 under the original measure. Finally, we provide two applications: 
(i) in Game Theory on the minmax guarantee levels of the players in games with infinitely many players, and
(ii) in Decision Theory on the set of feasible expected payoffs in infinite duration problems.
\end{abstract}

\smallskip


\section{Introduction}

\noindent\textbf{Approximations of expected payoffs.} Suppose that $\sigma = \otimes_{i \in \N} \sigma_i$ is a product probability measure on a product space $X = \times_{i \in \N} X_i$, where each $X_i$ is a Polish space,
and $f:X \to \R$ is a bounded and Borel-measurable function.
In this paper, 
we ask whether $\E_\sigma[f]$ can be approximated by an expectation $\E_\tau[f]$, 
for some ``simple'' product measure $\tau$. 

If $f$ is a tail function, then there necessarily exists $x \in X$ such that $f(x) = \E_\sigma[f]$; in fact, almost every (under $\sigma$) point of $X$ has the stated property. In general, no such point $x$ need to exist.

We are interested in two types of approximating measures.
We say that $x = (x_i)_{i \in \N}\in X$ is a \emph{weak $\ep$-approximation} of the expectation of $f$ under $\sigma$ if there is some product probability measure $\tau = \otimes_{i \in \N} \tau_i$,
such that (i) $| \E_\tau[f]-\E_\sigma[f]| \leq \ep$,
and (ii)
$\tau_i$ is the Dirac measure on $x_i$ for all but finitely many $i$'s.
We will say that 
$x\in X$ is a \emph{strong $\ep$-approximation} of the expectation of $f$ under $\sigma$ if the same holds, where in addition (iii) $\tau_i$ is either $\sigma_i$ or a Dirac measure, for every $i \in \N$.

Our main results are that (i) 
the set of weak 0-approximations have $\sigma$-measure 1 (cf. Theorem \ref{thm.exact}), and (ii) the set of strong $\ep$-approximations have $\sigma$-measure 1 for every $\ep>0$ (cf. Theorem \ref{thm.approx}). 
As an example demonstrates (cf. Example \ref{exno0}), a strong 0-approximation does not always exist.

\smallskip

\noindent\textbf{Application in Game Theory.} The first motivation for our study comes from Game Theory, and more specifically, from the study of strategic-form games with countably many players. Games with countably many players have been studied, e.g., by Peleg \cite{Peleg}, Voorneveld \cite{Voorneveld}, Rachmilevitch \cite{Rachmilevitch}, and Ashkenazi-Golan, Flesch, Predtetchinski and Solan \cite{Ashkenazi-etal}. These studies have mainly revolved around the question of (non)existence of Nash ($\ep$-)equilibrium. In contrast, here we focus on the notion of the minmax value.

The \textit{minmax value} of a player is one of the central concepts in game theory. It is defined as the maximal payoff that the player can guarantee to obtain, when all other players try to decrease her payoff. The underlying assumption on the behaviour of the punishing players is that they cannot correlate their lotteries, and thus their joint strategy is a product measure.

Ashkenazi-Golan, Flesch, Predtetchinski and Solan \cite{Ashkenazi-etal} introduced a notion of the \emph{finitistic minmax value}. This is a version of the minmax where only finitely many of the punishing players are allowed to randomize; the others are required to choose a Dirac measure. It is clear that the finitistic minmax value is not lower than the minmax value. Whether the two values are in fact equal is an open problem. Here we answer this question in the affirmative (cf. Theorem \ref{minmax-oneshot}). This is achieved using the result (ii) on strong $\ep$-approximations.\medskip 

\noindent\textbf{Application in Decision Theory.}  
The second motivation comes from decision theory. In dynamic decision problems on the infinite horizon, \emph{Markov strategies} play an important role (e.g., Hill \cite{Hill1979}, Kiefer, Mayr, Shirmohammadi, Totzke \cite{Kieferetal2019}, Sudderth \cite{Sudderth2020}). These are strategies in which, at each stage, the choice of action may depend on the current stage and the current state, but not on past events. 
Such strategies are desirable because of their simplicity,
as the decision maker's behavior depends only on the current payoff relevant parameters.
A strategy is called \emph{eventually pure} if after some stage it does not use randomization (cf. Flesch,  Herings, Maes, Predtetchinski \cite{Fleschetal2022}). 

Using our result (i) on weak 0-approximations, we show that in a class of decision problems on the infinite horizon, the set of expected payoffs that the decision maker can obtain by Markov strategies is equal to the set of expected payoffs under eventually pure Markov strategies. In other words, using Markov strategies that involve a randomization at infinitely many stages does not generate additional expected payoffs.\smallskip

\noindent\textbf{Structure of the paper.} 
Definitions appear in Section~\ref{section:definitions},
the main results and their proofs are provided in Section~\ref{section:results},
and the applications to Game Theory and Decision Theory are discussed in Section~\ref{section:application}. A few concluding remarks are given, including an open problem, in Section \ref{ConcRem}.

\section{Approximations of expectations}
\label{section:definitions}

Let $\N = \{1,2,\ldots\}$. For each $i \in \N$, let $X_{i}$ be a Polish space. Let $X = \times_{i \in \N} X_{i}$ denote the product space, which we endow with the product topology. 

For each $i \in \N$, let $\Sigma_{i}$ be the set of Borel probability measures on $X_{i}$, and let $\Sigma$ be the set of all product Borel probability measures on $X$, i.e., the measures $\sigma = \otimes_{i \in \N} \sigma_{i}$ where $\sigma_{i} \in \Sigma_{i}$ for each $i\in\N$. With a slight abuse of notation, we write $x_{i}$ to denote both a point of $X_{i}$ and the Dirac measure concentrated on $x_{i}$, and we use the same symbol $\sigma$ to denote both a measure in $\Sigma$ and its unique extension to the sigma-algebra of all universally measurable sets of $X$. 

We define two notions for the approximations of expectations under the measures in $\Sigma$. 
\begin{defn}
Let $\epsilon \geq 0$. For a product measure $\sigma \in \Sigma$ and a bounded Borel-measurable function $f : X \to \R$, a point $x \in X$ is said to be a \emph{weak $\epsilon$-approximation of the expectation of $f$ under $\sigma$} if there is $n \in \N$ and a Borel probability measure $\tau_{i}$ on $X_{i}$ for each $i<n$, such that: 
\begin{equation}\label{eqn.1player}
\big|\E_{\tau_{1} \otimes \cdots \otimes \tau_{n-1} \otimes x_{n} \otimes x_{n+1} \otimes \cdots}[f] - \E_{\sigma}[f]\big| \,\leq\, \epsilon.
\end{equation}  
\end{defn}

Intuitively, \Eqref{eqn.1player} means that the expectation changes by at most $\ep$ if we replace $\sigma$ by the measure that is the Dirac measure on $x_i$ in each coordinate $i\geq n$ and $\tau_i$ in each coordinate~$i<n$. 

\begin{defn}
Let $\epsilon \geq 0$. For a product measure $\sigma \in \Sigma$ and a bounded Borel-measurable function $f : X \to \R$, a point $x \in X$ is said to be a \emph{strong $\epsilon$-approximation of the expectation of $f$ under $\sigma$} if there is $n \in \N$ such that: 
\begin{equation}\label{eqn.apprx}
\big|\E_{\sigma_{1} \otimes \cdots \otimes \sigma_{n-1} \otimes x_{n} \otimes x_{n+1} \otimes \cdots}[f] - \E_{\sigma}[f]\big| \,\leq\, \epsilon.
\end{equation}
\end{defn}

Thus, \Eqref{eqn.apprx} means that the expectation changes by at most $\ep$ if we replace $\sigma$ by the Dirac measure on $x_i$ in each coordinate $i\geq n$. Obviously, if $x \in X$ is a strong $\epsilon$-approximation of the expectation of $f$ under $\sigma$, then $x$ is also a weak $\epsilon$-approximation.

Let $W_{\epsilon}(\sigma,f) \subseteq X$ and $S_{\epsilon}(\sigma,f) \subseteq X$ denote the sets of weak and respectively strong $\epsilon$-approximations of the expectation of $f$ under $\sigma$. 
Thus, $W_{\epsilon}(\sigma,f) \subseteq S_{\epsilon}(\sigma,f)$. It follows from the following lemma that both $W_{\epsilon}(\sigma,f)$ and $S_{\epsilon}(\sigma,f)$ are universally measurable, and in particular, both sets are measurable with respect to the measure $\sigma$.

\begin{lemma}\label{lemAn}
For each $\ep\geq 0$, each product measure $\sigma\in \Sigma$, and each bounded Borel-measurable function $f : X \to \R$, the set $W_{\epsilon}(\sigma,f) \subseteq X$ is analytic and the set $S_{\epsilon}(\sigma,f) \subseteq X$ is Borel.
\end{lemma}

\begin{proof}
For each $n\in\N$, we equip $\Sigma_{n}$ with the topology of weak convergence, and $\Sigma^n:=\times_{i=1}^n\Sigma_{i}$ with the product topology. Note that $\Sigma^n$ is a Polish space. 
For each $n \in \N$, the function $X \times \Sigma^n \to \R$ given by $(x,\tau_{1},\ldots,\tau_{n}) \mapsto \E_{\tau_{1} \otimes \cdots \otimes \tau_{n-1} \otimes \tau_{n} \otimes x_{n+1} \otimes x_{n+2} \otimes \cdots}[f]$ is Borel-measurable (Aliprantis and Border \cite[Theorem 15.13]{AliBo06}). Hence the set
$B^{n} \subseteq X \times \Sigma^{n}$ consisting of the points $(x,\tau_{1},\ldots,\tau_{n})$ satisfying \Eqref{eqn.1player} is a Borel set. Let $A^{n}$ be the projection of $B^{n}$ onto the first coordinate. Then $A^{n}$ is analytic. Since $W_{\epsilon}(\sigma,f) = \cup_{n \in \N} A^{n}$, the result follows.

The fact that $S_{\epsilon}(\sigma,f)$ is a Borel set follows since the map $X \to \R$ given by $x \mapsto \linebreak \E_{\sigma_{1} \otimes \cdots \otimes \sigma_{n-1} \otimes x_{n} \otimes x_{n+1} \otimes \cdots}[f]$ is Borel-measurable for each $n\in\N$ (Aliprantis and Border \cite[Theorem 15.13]{AliBo06}).
\end{proof}

\section{Results}
\label{section:results}

\begin{thm}\label{thm.exact}
For each product measure  $\sigma \in \Sigma$ and each bounded Borel-measurable function $f : X \to \R$, the set $W_{0}(\sigma,f)$ has $\sigma$-measure $1$. 
\end{thm}

\begin{proof}
Fix a product measure  $\sigma \in \Sigma$ and a bounded Borel-measurable function $f : X \to \R$. 

Let $\sim$ be the following tail equivalence relation: for $x,y \in X$ let $x \sim y$ if there is some $n \in \N$ such that $x_{i} = y_{i}$ for all $i \geq n$. Let $\mathscr{Z}$ denote the collection of all equivalence classes of $\sim$. 
For $Z \in \mathscr{Z}$, let $I_{Z}$ be the convex hull of $f(Z)=\{f(x):x \in Z\}$; in particular, $I_{Z}$ is an interval, with or without its endpoints. 

Let $r = \E_{\sigma}[f]$. Define 
\begin{align*}
Z_{-} &= \bigcup\{Z \in \mathscr{Z}: I_{Z} \subseteq (-\infty,r)\},\\
Z_{0}\ &= \bigcup\{Z \in \mathscr{Z}: I_{Z} \ni r\},\\
Z_{+} &= \bigcup\{Z \in \mathscr{Z}: I_{Z} \subseteq (r, +\infty)\}.
\end{align*}

The statement of Theorem \ref{thm.exact} follows from Claims~1 and 2 below.\smallskip

\noindent\textsc{Claim 0:} $Z_{-}$ and $Z_{+}$ are co-analytic subsets of $X$.

\smallskip

It follows from Claim~0 that $Z_{0}$ is analytic,
and hence $Z_-$, $Z_0$, and $Z_+$ are measurable with respect to $\sigma$. \smallskip

\noindent\textsc{Proof of claim 0:} Let us write $\sim$ to denote the set $\{(x,y) \in X \times X: x \sim y\}$.
Then $\sim$ is a Borel subset of $X \times X$, as can be seen by rewriting it as
\[\bigcup_{n \in \N} \bigcap_{i \geq n}\{(x,y) \in X \times X: x_{i} = y_{i}\}.\] 
The set $X \setminus Z_{-}$ is a projection of the Borel set 
\[\sim \bigcap \{(x,y) \in X \times X: f(y) \geq r\}\] 
onto its first coordinate, and therefore analytic.
The argument that $X \setminus Z_{+}$ is analytic is similar. \medskip 

\noindent\textsc{Claim 1:} $Z_{0} \subseteq W_{0}(\sigma,f)$.\smallskip

\noindent\textsc{Proof of Claim 1:} Take any $Z \in Z_0$. We show that $Z\subseteq W_{0}(\sigma,f)$.

As $Z \in Z_0$, we have $r \in I_{Z}$, and there exist points $x,y \in Z$ such that $f(x) \leq r \leq f(y)$. Let $n \in \N$ be such that $x_{i} = y_{i}$ for each $i > n$.

For each $k=1,\ldots,n+1$, let $z_k=(y_1,\ldots,y_{k-1},x_k,x_{k+1},\ldots)$. In particular, $z_1=x$ and $z_{n+1}=y$, and $z_k\in Z$ for each $k=1,\ldots,n+1$. Because $z_1=x$ and $z_{n+1}=y$, the interval $[f(x),f(y)]$ is covered by the union of the intervals $[f(z_k),f(z_{k+1})]$ with $k=1,\ldots,n$. Consequently, there is $k\in\{1,\ldots,n\}$ such that $r\in [f(z_k),f(z_{k+1})]$. Note that $z_k$ and $z_{k+1}$ only differ in coordinate $k$.

Let $\alpha\in[0,1]$ be such that $\alpha\cdot f(z_k)+(1-\alpha)f(z_{k+1})=r$. Define the measure $\tau_k$ on $X_k$ as follows: $\tau_k$ places probability $\alpha$ on $x_k$ and probability $1-\alpha$ on $y_k$. Then, \[\E_{y_{1} \otimes \cdots \otimes y_{k-1} \otimes \tau_{k} \otimes x_{k+1} \otimes x_{k+2} \otimes \cdots}[f]\,=\,\alpha\cdot f(z_k)+(1-\alpha)f(z_{k+1})\,=\,r\,=\,\E_\sigma[f].\]
This means that $z_k=(y_1,\ldots,y_{k-1},x_k,x_{k+1},\ldots)$ is a weak $0$-approximation of the expectation of $f$ under $\sigma$,
so that $z_k\in W_{0}(\sigma,f)$. 
Hence, so is each point in the equivalence class $Z$ of $z_k$.
\medskip

\noindent\textsc{Claim 2:} $Z_{0}$ has $\sigma$-measure 1.\smallskip

\noindent\textsc{Proof of Claim 2:} The sets $Z_{-}$, $Z_{0}$, and $Z_{+}$ are tail. Since $\sigma$ is a product measure, the $\sigma$-measure of each set is either $0$ or $1$ (see Lemma \ref{lemKolg} in the Appendix). 
Since $Z_{-}$, $Z_{0}$, and $Z_{+}$ partition $X$, exactly one of them has $\sigma$-measure 1. It cannot be the case that $\sigma(Z_{-}) = 1$, since then, as $Z_{-} \subseteq f^{-1}(-\infty,r)$, we would obtain $\E_{\sigma}[f] < r$, a contradiction. One similarly excludes the possibility that $\sigma(Z_{+}) = 1$. Hence $\sigma(Z_{0}) = 1$.
\end{proof}

\begin{rem}\rm The proof of Claim 1 in the proof of Theorem \ref{thm.exact} reveals the following statement: For each product measure $\sigma\in\Sigma$ and each bounded Borel-measurable function $f : X\to\mathbb{R}$, there exists a weak 0-approximation $x\in W_0(\sigma,f)$ with the following property: there is a coordinate $n$ and a Borel probability measure $\tau_n$ on $X_n$ such that  \[\E_{x_{1} \otimes \cdots \otimes x_{n-1} \otimes \tau_{n} \otimes x_{n+1} \otimes x_{n+2} \otimes \cdots}[f]\,=\,\E_\sigma[f].\] That is, we only need to replace $x$ in a single coordinate by a Borel probability measure.
\end{rem}

\begin{thm}\label{thm.approx}
For each product measure  $\sigma \in \Sigma$, each bounded Borel-measurable function $f : X \to \R$, and each $\epsilon > 0$, the set $S_{\epsilon}(\sigma,f)$ has $\sigma$-measure $1$.  
\end{thm}
\begin{proof}
Let $\G_{n}$ be the sigma-algebra on $X$ generated by the coordinate functions $x_{k}$, $k \geq n$. Let $\G = \bigcap_{n \in \mathbb{N}}\G_{n}$. This is the tail sigma-algebra on $X$. Consider $g_{n} = \E_{\sigma}[f|\G_{n}]$. 

The process $\{g_{n}\}_{n \in \mathbb{N}}$ is a reverse martingale with respect to the sequence  $\{\G_{n}\}_{n \in \mathbb{N}}$ and the measure $\sigma$. 
That is, $\E_{\sigma}[g_{n} \mid \G_{n+1}] = g_{n+1}$. 
By Durrett \cite[Theorems 5.6.1 and 5.6.2]{Durrett11}, the sequence $g_{n}$ converges to a $\G$-measurable limit $g_{\infty}$, almost surely with respect to $\sigma$, and $\E_{\sigma}[g_{\infty}] = \E_{\sigma}[g_0] = \E_{\sigma}[f]$. Since $\sigma$ is a product measure, Kolmogorov's 0-1 law implies that $g_{\infty}$ is constant $\sigma$-almost surely. Hence $\sigma$-almost surely, $g_{n}$ converges to $\E_{\sigma}[f]$. This means that $\sigma$-almost surely, \Eqref{eqn.apprx} holds for large $n \in \N$, as desired.
\end{proof}

As the following example shows, there are product measures $\sigma\in\Sigma$ and bounded Borel-measurable functions $f$ such that there is no strong $0$-approximation of the expectation of $f$ under $\sigma$, i.e., $S_{0}(\sigma,f)=\emptyset$.

\begin{exl}\label{exno0}\rm
Let $X_{i} = \{0,1\}$ for each $i \in \N$, and let $f(x)$ be 1 if $x_{i} = 1$ for each $i \in \N$ and $0$ otherwise. We identify $\sigma_{i} \in \Sigma_{i}$ with $\sigma_{i}(\{1\})$, the probability of 1. Take any $\sigma \in \Sigma$ such that (a) $0 < \E_{\sigma}[f]$ and (b) there are infinitely many $i \in \N$ such that $\sigma_{i} < 1$ (for example $\sigma_{i} = 1 - 2^{-i}$ for each $i \in \N$). Then, $S_{0}(\sigma,f)=\emptyset$. Indeed, consider an $x \in X$ and an $n \in \N$. If $x_{i} = 0$ for some $i \geq n$, then $\E_{\sigma_{1} \otimes \cdots \otimes \sigma_{n-1} \otimes x_{n} \otimes x_{n+1} \otimes \cdots}[f] = 0$. If, on the other hand, $x_{i} = 1$ for each $i \geq n$, then 
\[\E_{\sigma}[f] = \prod_{i = 1}^{\infty}\sigma_{i} < \prod_{i = 1}^{n-1}\sigma_{i} = \E_{\sigma_{1} \otimes \cdots \otimes \sigma_{n-1} \otimes x_{n} \otimes x_{n+1} \otimes \cdots}[f].\]
 \end{exl}
\section{Applications in Decision Theory and Game  Theory}
\label{section:application}

In this section, we provide two applications of our results: one in Game Theory and one in Decision Theory.\medskip

\noindent\textbf{Application in Game Theory.} We apply
Theorem~\ref{thm.approx} on strong $\ep$-approximations of expectations 
to prove a result on the minmax values of the players in strategic-form games with an infinite set of players.

We consider a strategic-form game in which the set of players is $\{0\}\cup\N$. Player $0$'s action set is a nonempty finite set $A$, whereas the action set of each player $i \in \N$ is a Polish space $X_{i}$. As before, $X=\times_{i\in\N}X_i$. 
Player $0$'s payoff function is $u : A \times X \to \R$, a bounded Borel-measurable function. 
As before, $\Sigma$ denotes the set of product probability measures on $X$, which is the set of strategy profiles of player $0$'s opponents. Let $\Sigma^{*} \subseteq \Sigma$ denote the set of strategy profiles $\sigma = \otimes_{i \in \N} \sigma_{i}$ of player $0$'s opponents such that $\sigma_{i}$ is pure (i.e., a Dirac measure on $X_{i}$) for all but finitely many $i \in \N$.

The \textit{minmax value} and the \textit{finitistic minmax value} of player $0$ are defined respectively as
\begin{align*}
v =& \inf_{\sigma \in \Sigma} \max_{a \in A} \E_{a \otimes \sigma}[u],\\
v^{*} =& \inf_{\sigma \in \Sigma^{*}} \max_{a \in A} \E_{a \otimes \sigma}[u].
\end{align*}

Intuitively, the minmax value of player 0 is the maximal expected payoff that player 0 can guarantee to obtain, when all other players try to lower her payoff. The minmax value is a fundamental concept, and the expected payoff of a player in a Nash equilibrium can never be below her minmax value. 

The finitistic minmax value is a variant where only finitely many opponents can randomize over their actions. As mention in the introduction, the concept of the finitistic minmax value played a crucial role in Ashkenazi-Golan, Flesch, Predtetchinski and Solan \cite{Ashkenazi-etal} in their study of games with infinitely many players.

It follows directly from the definitions that $v \leq v^{*}$. By using Theorem~\ref{thm.approx}, it turns out, perhaps surprisingly, that in fact equality holds. 

\begin{thm}\label{minmax-oneshot} 
$v = v^{*}$. 
\end{thm}

\begin{proof}
Since $\Sigma^* \subseteq \Sigma$, 
we always have $v \leq v^*$.
To prove Theorem~\ref{minmax-oneshot} we therefore need to show that $v^* \leq v$.
Fix $\epsilon > 0$. We show that $v^{*} \leq v + 2\epsilon$.

For each $a \in A$, let $f_{a} = u(a,\cdot): X \to \R$. Let $\sigma \in \Sigma$ be the $\epsilon$-minmax strategy profile of player $0$'s opponents; 
that is, a strategy profile such that $ \E_{a \otimes \sigma}[u] = \E_{\sigma}[f_{a}]\leq v + \epsilon$ for each $a \in A$. 

Take any $x \in \bigcap_{a \in A}S_\epsilon(f_{a},\sigma)$, which is nonempty by Theorem~\ref{thm.approx} and since $A$ is finite. 
Using once again the finiteness of $A$, 
there exists $n \in \N$ such that \Eqref{eqn.apprx} holds for each of the functions $f_{a}$, and thus $\E_{\sigma_{1} \otimes \cdots \otimes \sigma_{n-1} \otimes x_{n} \otimes x_{n+1} \otimes \cdots}[f_a] \,\leq\, \E_{\sigma}[f_a] +  \epsilon$ for each $a\in A$.
This implies that the strategy profile $\tau=\sigma_{1} \otimes \cdots \otimes \sigma_{n-1} \otimes x_{n} \otimes x_{n+1} \otimes \cdots \in \Sigma^{*}$ of player $0$'s opponents is a $2\epsilon$-minmax strategy profile: $ \E_{a \otimes \tau}[u] = \E_{\tau}[f_{a}]\leq \E_{\sigma}[f_a] +  \epsilon\leq v + 2\epsilon$ for each $a \in A$. Therefore, $v^{*} \leq v + 2\epsilon$, as desired. 
\end{proof}

As the following example shows, 
the statement of Theorem~\ref{minmax-oneshot} is not true if player $0$'s action set $A$ is infinite. 

\begin{exl}\rm 
Suppose that $A = \N \times \{0,1\}$ and $X_{i} = \{0,1\}$ for each $i \in \N$. Define player $0$'s payoff function as follows: for $(n,j) \in A$ and $x \in X$, let $u((n,j),x)$ be $1$ if $j = x_{n}$ and $0$ otherwise. Intuitively, player 0's action $(n,j)$ can be thought of naming an opponent $n$ and an action $j$, and if opponent $n$'s action is indeed $j$, then player 0 obtains payoff 1. 

Then $v = \tfrac{1}{2}$, which follows by noticing that if $\sigma_{i}$ places equal probability on $0$ and $1$ for each $i \in \N$ then $\E_{(n,j) \otimes \sigma}[u]=\tfrac{1}{2}$ for each $(n,j)\in A$. On the other hand, $v^*=1$, as in any strategy profile $\sigma\in\Sigma^*$ there is a player $n\in\mathbb{N}$ who places probability 1 on one of his actions $j\in\{0,1\}$, and then $\E_{(n,j) \otimes \sigma}[u]=1$.
\end{exl}

\noindent\textbf{Application in Decision Theory.} We apply Theorem \ref{thm.exact} to the the following class of decision problems on the infinite horizon: At each stage $t\in\N$, a decision maker chooses an action $x_t$ from a Polish space $X_t$. If the sequence of chosen actions is $(x_1,x_2,\ldots)\in X=\times_{t\in\N}X_t$, then her payoff is $u(x_1,x_2,\ldots)$, where $u : X \to \R$ is a bounded Borel-measurable function. 

A \emph{Markov strategy} is an element of $\Sigma$, which is, as before, the set of product probability measures on $X$. Intuitively, a Markov strategy $\sigma=(\sigma_t)_{t\in\N}\in\Sigma$ chooses an action at each stage $t$ only depending on the current stage $t$, and thus not depending on the past actions. Such a strategy is called \emph{eventually pure}, if $\sigma_t$ is a Dirac measure for all but finitely many stages $t\in \N$.

It follows from Theorem \ref{thm.exact} that the set of expected payoffs that the decision maker can obtain by Markov strategies is equal to the set of expected payoffs under eventually pure Markov strategies. In other words, randomizing infinitely many times in a Markov strategy does not give additional expected payoffs to the decision maker.\medskip

\section{Concluding Remarks}\label{ConcRem} 
We do not know whether Theorem \ref{thm.exact} could be deduced from Theorem  \ref{thm.approx} or vice versa. The two results appear to be logically unrelated, especially as the proofs use different approaches.

We conclude with an open problem: Can one generalize Theorem \ref{thm.exact} to vector-valued (rather than just real-valued) functions? Consider a bounded Borel-measurable function $f : X \to \R^{d}$, and a product measure $\sigma \in \Sigma$. Is there a point $x \in X$, an $n \in \N$, and probability measures $\tau_{i} \in \Sigma_{i}$ for $i < n$, such that 
\[\E_{\tau_{1} \otimes \cdots \otimes \tau_{n-1} \otimes x_{n} \otimes x_{n+1} \otimes \cdots}[f_{i}] = \E_{\sigma}[f_{i}]\]
for each $i \in \{1,\ldots,d\}$? Theorem \ref{thm.exact} gives the affirmative answer when $d = 1$. The answer is also a ``yes" when $f$ is a tail function (i.e., when $f$ is measurable with respect to the tail sigma-algebra on $X$). Beyond these special cases, the problem is open. In particular, it is also open if $f$ is continuous. 

\section*{Appendix}
We have applied Kolmogorov's zero-one law to analytic and co-analytic sets. Here we justify this application.

\begin{lemma}\label{lemKolg}
If $U$ is a universally measurable tail subset of $X$ and $\sigma \in \Sigma$, then the $\sigma$-measure of $U$ can only have the values 0 and 1.  
\end{lemma}
\begin{proof}
We adapt the proof of the Kolmogorov's zero-one law in Shiryaev \cite[p.381]{Shiryaev96}.

For $n \in \N$ let $\mathscr{F}_{n}$ be the sigma-algebra on $X$ generated by the first $n$ coordinate functions $x_{1},\ldots,x_{n}$. 
The Borel sigma-algebra $\mathscr{F}$ on $X$ is the sigma-algebra generated by the set $\cup_{n \in \N}\mathscr{F}_{n}$. As $U$ is universally measurable, there exists a Borel set $B$ of $X$ such that $\sigma(B \triangle U) = 0$, where $\triangle$ denotes the symmetric difference. 

By Shiryaev \cite[Problem 8, p.69]{Shiryaev96},
for every $n \in \N$ there exists a set $B_{n} \in \mathscr{F}_{n}$ such that $\sigma(B_{n} \triangle B) \to 0$. Since $\sigma$ is a product measure and $U$ is tail, $B_{n}$ and $U$ are independent. 
Therefore,
\[\sigma(B_{n}) \to \sigma(B) = \sigma(U)\ \,\text{ and}\]
\[\sigma(B_{n}) \cdot \sigma(U) = \sigma(B_{n} \cap U) = \sigma(B_{n} \cap B) \to \sigma(B) = \sigma(U).\]  
It follows that $\sigma(U) = \sigma(U) \cdot \sigma(U)$, implying the result.
\end{proof}

\end{document}